\documentclass[11pt]{article}

\usepackage{fullpage}
\usepackage{amsmath, amsthm, amsfonts, amssymb, amstext, mathrsfs, enumerate}
\usepackage{graphicx, ragged2e, lscape, framed, xcolor}
\usepackage{subfiles}

\theoremstyle{plain}
\newtheorem{theorem}{Theorem}[section]

\newtheorem{proposition}[theorem]{Proposition}
\newtheorem{conjecture}[theorem]{Conjecture}

\newtheorem{problem}[theorem]{Problem}
\newtheorem{claim}{Claim}[section]

\numberwithin{equation}{section}
\allowdisplaybreaks

\usepackage[pagebackref]{hyperref}
\hypersetup{
	colorlinks=true,
    urlcolor=purple,
	linkcolor=purple,
    citecolor=purple,
}

\DeclareMathOperator{\spec}{spec}
\DeclareMathOperator{\supp}{supp}

\def\b{\mbox{\boldmath $b$}}
\def\s{\mbox{\boldmath $s$}}

\def\x{\mbox{\boldmath $x$}}
\def\y{\mbox{\boldmath $y$}}

\def\1{\mbox{\boldmath $1$}}
\def\0{\mbox{\boldmath $0$}}

\newcommand{\affl}[3]{\noindent #1, Email: {\tt #2}\\ \textsc{#3}\\[1pt]}

\title{\textbf{The switching conjecture for main eigenvalues is asymptotically true}}
\author{Saieed Akbari \and Hitesh Kumar \and Bojan Mohar \and Shivaramakrishna Pragada}
\date{}

\begin{document}
\maketitle
\begin{abstract} An eigenvalue of a signed graph is called \emph{main} if there exists a corresponding eigenvector non-orthogonal to the all-ones vector. An important result of O'Rourke and Touri (2016) states that almost all (unsigned) graphs have all main eigenvalues. Akbari, Fran\c{c}a, Ghasemian, Javarsineh, and de Lima (2021) considered main eigenvalues of signed graphs and conjectured that for any unsigned connected graph $G \notin\{ K_2, K_4 - e\}$, there is a switching $\s$ such that all eigenvalues of the signed graph $G^{\s}$ are main.

We prove two incomparable asymptotic versions of this conjecture. We show that for any graph $G$ of order $n$, there exists a switching
$\s\in\{\pm1\}^n$ such that $G^{\s}$ has $n - O\!\left(\frac{n}{(\log n)^{1/4}}\right)$ main eigenvalues counted with multiplicity. Using a similar proof strategy, we also show that if $G$ has $d$ distinct eigenvalues, then there exists a switching $\s\in\{\pm1\}^n$ such that $G^{\s}$ has $d - O\!\left(\frac{d}{(\log d)^{1/4}}\right)$ main eigenvalues.
\end{abstract}



\section{Introduction}

\subsection{Notation and terminology}

Throughout, assume that all (signed) graphs are of order $n\ge 5$ unless stated otherwise. A \emph{signed graph} $\Gamma$ is a pair $(G, \sigma)$, where $G =(V(G), E(G))$ is a simple graph, called the \emph{underlying unsigned graph}, and $\sigma:E(G)\rightarrow\{-1, +1\}$ is the \emph{signature}. The \emph{adjacency matrix} of $\Gamma$, denoted by $A(\Gamma)$, is defined as
\[A(\Gamma)_{uv}=
\begin{cases}
    \sigma(uv)&\text{ if } uv\in E(G);\\
    0&\text{ otherwise.}
\end{cases}\]
An unsigned graph $G$ is taken to be equal to the signed graph $(G, 1)$ where $1$ is the signature that assigns $1$ to all edges of $G$. For a signed graph $\Gamma$ with $V(\Gamma) = \{1, \ldots, n\}$, a function $\s: V(\Gamma) \to \{1 ,-1\}$ is called a \textit{switching}. Equivalently, a switching is a Rademacher vector $\s = (s_1, \ldots, s_n)\in \{\pm 1\}^n$. Two signed graphs $\Gamma_1 = (G, \sigma_1)$ and $\Gamma_2 = (G, \sigma_2)$ with the same underlying graph $G$ are \textit{switching equivalent}, written $\Gamma_1 \sim \Gamma_2$, if there is a switching $\s$ such that \[\sigma_2(ij)=s_i\, \sigma_1(ij)\, s_j.\]
This is equivalent to
\begin{equation}\label{eq:switching_equivalence}
    A(\Gamma_2) = D_{\s}\, A(\Gamma_1)\, D_{\s},
\end{equation}
where $D_{\s}$ is the diagonal matrix whose diagonal vector is $\s$. Using \eqref{eq:switching_equivalence} it is clear that $\lambda$ is an eigenvalue of $A(\Gamma_1)$ with eigenvector $\x$ if and only if $\lambda$ is an eigenvalue of $A(\Gamma_2)$ with eigenvector $D_{\s} \x$. In particular, switching equivalent signed graphs have the same spectrum. For further details on signed graphs, refer to \cite{Harary_1955, Zaslavsky_1982, Zaslavsky_1998}. For an unsigned graph $G$ and a switching $\s$, we denote by $G^{\s}$\ the signed graph with signature $\sigma$ given by
\[ \sigma(ij) = s_is_j\]
whenever $ij\in E(G)$. For a subset $U\subseteq V(G)$, we will denote by $\s_U$ the switching whose entries are given by
\[ s_v: = 
\begin{cases}
-1 & v\in  U;\\
1 & v\notin U.
\end{cases}\]
If $U=\{v\}$ (resp. if $U = \{u,v\}$), then we write $\s_v$ (resp. $\s_{u,v}$) instead of $\s_U$.
    
Consider a signed graph $\Gamma = (G, \sigma)$ with adjacency matrix $A(\Gamma)$. Since $A(\Gamma)$ is a real symmetric matrix, the eigenvalues of $A(\Gamma)$ (or simply $\Gamma$) are real which we denote as:
\[ \lambda_1(\Gamma)\ge \cdots \ge \lambda_n(\Gamma).\]
For an eigenvalue $\lambda$ of $\Gamma$, we denote by $E_{\lambda}$ the \emph{eigenspace} spanned by the eigenvectors of $A(\Gamma)$ corresponding to $\lambda$. We will denote by $\spec(\Gamma)$ the set of \emph{distinct} eigenvalues of $\Gamma$.  

Throughout, $\1$, $I$, and $\0$ will denote the all-ones vector, the identity matrix, and the zero vector of appropriate order. For two vectors $\x, \y\in \mathbb{R}^n$, we denote by $\x^\top\y$ the usual dot product of two vectors. 

Let $\x = (x_1, \ldots, x_n)\in \mathbb{R}^n$. The \emph{support} of $\x$ is defined to be the set $\supp(\x) := \{i: x_i\ne 0\}$. We say that the vector $\x$ is \emph{non-main} if $\1^\top\x = 0$ (equivalently, $\x \in \1^{\perp}$), and \emph{main} otherwise. An eigenvalue $\lambda$ of a signed graph $\Gamma$ is called \emph{non-main} if $E_{\lambda}\subseteq \1^{\perp}$ and \emph{main} otherwise. 

\subsection{Background and motivation}

For a graph $G$, its \emph{walk matrix} $W(G)$ is defined to be \[W(G) = [\ \1\ |\ A\1\ |\ \cdots\ |\ A^{n-1}\1\ ],\] 
where $A = A(G)$. The walk matrix is a special case of so-called \emph{controllability matrices} which are well studied in the control theory literature. There is a deep connection between the main eigenvalues and the rank of the walk matrix, see an early paper by Cvetkovi\'{c} \cite{Cvetkovic_1978}. It is known, for instance, that the number of distinct main eigenvalues of $G$ equals the rank of its walk-matrix $W(G)$ (see \cite{Teranishi_2001, Hagos_2002} and the survey \cite{Rowlinson_2007}). A graph $G$ of order $n$ is called \emph{controllable} if $W(G)$ has full rank, equivalently, if $G$ has $n$ distinct main eigenvalues. For applications of controllable graphs in Control Theory, see \cite{Kailath_1980, Ogata_2010, Rahmani_Ji_Mesbahi_Egerstedt_2009, Cvetkovic_Rowlinson_Stanic_Yoon_2014_multiagent}. Motivated by these applications, Cvetkovi\'{c}, Rowlinson, Stani\'{c}, and Yoon \cite{Cvetkovic_Rowlinson_Stanic_Yoon_2011, Cvetkovic_Rowlinson_Stanic_Yoon_2011_least_eigenvalue} (also see \cite{Stanic_2014}) investigated them in a more theoretical setting. 

Controllable graphs are interesting in the context of reconstructibility, see the results of Godsil and McKay \cite{Godsil_McKay_1981} (cf. \cite{He_2007}) and Yuan \cite{Yuan_1982}. They also tend to be good candidates for graphs determined by their spectra. An important result of Wang \cite{Wang_2017} shows that controllable graphs that satisfy a certain arithmetic condition are determined by their generalized spectrum. This has led to fruitful research, see the citations in and of \cite{Wang_2017}. Also, see \cite{Lou_Huang_Huang_2017} and its citations for some methods of constructing controllable graphs with certain properties. 

Godsil \cite{Godsil_2012} conjectured that almost all graphs are controllable. This was later confirmed by O'Rourke and Touri \cite{Rourke_Touri_2016}. Note that not all graphs are controllable; in fact, they can be very far from it. For instance, regular graphs have only one main eigenvalue. Stani\'{c} \cite{Stanic_2020} initiated the study of the main eigenvalues of signed graphs characterizing signed graphs with exactly one main eigenvalue. Also see the recent work on signed graphs with two distinct main eigenvalues \cite{Du_You_Liu_Yuan_2024, Du_You_Liu_2026}.

While not all graphs have all main eigenvalues, Akbari, Fran\c{c}a, Ghasemian, Javarsineh, and de Lima \cite{Akbari_Franca_Ghasemian_Javarsineh_deLima_2021} wondered: for a given graph $G$, does there always exist a switching $\s$ such that all (distinct) eigenvalues of $G^{\s}$ are main? After some initial investigation, they proposed the following conjecture.

\begin{conjecture}[Switching conjecture for main eigenvalues \cite{Akbari_Franca_Ghasemian_Javarsineh_deLima_2021}]\label{conj:switching_n_main_eigenvalues}
For any unsigned connected graph $G \notin\{ K_2, K_4 - e\}$, there is a switching $\s$ such that all (distinct) eigenvalues of $G^{\s}$ are main.
\end{conjecture}

Note that it suffices to work with \emph{distinct} eigenvalues. Indeed, if $\lambda$ is a main eigenvalue of $G^{\s}$, then one can always find a basis of its eigenspace $E_{\lambda}$ which only contains main $\lambda$-eigenvectors (see the proof of Proposition \ref{prop:equivalence}). In other words, if $\lambda$ is a main eigenvalue with multiplicity $m_{\lambda}$, then one can find $m_{\lambda}$ main $\lambda$-eigenvectors. Conjecture \ref{conj:switching_n_main_eigenvalues} is indeed equivalent to the following seemingly stronger (and unambiguous) assertion.

\begin{conjecture}\label{conj:switching_n_main_eigenvectors}
For any unsigned connected graph $G \notin\{ K_2, K_4 - e\}$ of order $n$, there is a  switching $\s$ such that $G^{\s}$ has an orthonormal eigenbasis consisting of main eigenvectors. 
\end{conjecture}    

For completeness, let us argue the equivalence of the above two conjectures.

\begin{proposition}\label{prop:equivalence}
Conjectures \ref{conj:switching_n_main_eigenvalues} and \ref{conj:switching_n_main_eigenvectors} are equivalent.
\end{proposition}

\begin{proof}
We only need to argue that Conjecture \ref{conj:switching_n_main_eigenvalues} implies Conjecture \ref{conj:switching_n_main_eigenvectors}; the other direction is clear. 

Assume that there is a switching $\s$ such that all distinct eigenvalues of $G^{\s}$ are main. Since $A(G^{\s})$ is symmetric, the eigenspaces corresponding to distinct eigenvalues are orthogonal. So it suffices to find an orthogonal basis of main eigenvectors for $E_{\lambda}$ for any given eigenvalue $\lambda$ of $G^{\s}$.

Let $\x_1$ denote the normalized orthogonal projection of $\1$ on $E_{\lambda}$. Since $E_{\lambda}\not \subseteq \1^\perp$, we have $\x_1 \ne \0$. Extend $\x_1$ to an orthonormal basis $\x_1, \ldots, \x_{m_{\lambda}}$ of $E_{\lambda}$ where $m_{\lambda}=\dim E_\lambda$. Clearly, $\1^\top \x_i = 0$ for all $2\le i\le m_\lambda$. Now, take an $m_\lambda\times m_\lambda$ orthogonal matrix $Q$ whose first row is
\[ \frac{1}{\sqrt{m_\lambda}}(1, \ldots, 1).\]
Such a matrix exists by extending the above row to an orthonormal basis of $\mathbb{R}^{m_\lambda}$. Define 
\[ \y_{j} = \sum_{i=1}^{m_\lambda} Q_{ij} \x_{i}.\]
Since $Q$ is orthogonal, $\y_1, \ldots, \y_{m_\lambda}$ form an orthonormal basis for $E_{\lambda}$. Moreover, we see that
\[ \1^\top \y_j = Q_{1j} \1^\top \x_1 = \frac{1}{\sqrt{m_\lambda}}\1^\top \x_1 \neq 0. \]
Thus, $\y_1, \ldots, \y_{m_\lambda}$ are main eigenvectors corresponding to $\lambda$ for $G^{\s}$. This completes the proof.
\end{proof}

In \cite{Akbari_Franca_Ghasemian_Javarsineh_deLima_2021}, the authors verified this conjecture for distance-regular graphs, vertex-transitive graphs, and some special graphs, including all connected graphs on at most 9 vertices excluding $K_2$ and $K_4-e$. Also, see \cite{Stanic_2020, Shao_Yuan_2022} for verification of Conjecture \ref{conj:switching_n_main_eigenvalues} in special cases. We note that a similar problem was considered for the Laplacian matrix in \cite{Andelic_Koledin_Stanic_2023}. 

Finally, note that connectedness of $G$ is a necessary assumption in Conjecture \ref{conj:switching_n_main_eigenvalues}. It is easily seen that $K_2 \cup r K_1$ $(r\ge 1)$ has no switching, which makes all eigenvalues main. For another non-trivial example, consider $K_2 \cup P_m$, where $P_m$ is a path with $m \not \equiv 2 \mod 3$. Indeed, in this case $P_m$ does not have $1$ or $-1$ as an eigenvalue, and so the support of any main eigenvector for $\pm 1$ is contained in $V(K_2)$. Thus, no switching of $K_2\cup P_m$ makes both $1$ and $-1$ main.

\subsection{Our contribution}

Given an unsigned graph $G$ of order $n$ and a switching $\s\in \{\pm 1\}^n$, define
\[
M(\s) := \sum_{\lambda\in \spec(G)} m_\lambda \mathbf{1}_{\{E_\lambda\not\subseteq \s^\perp\}}\quad \text{ and }\quad 
M_*(\s) := \sum_{\lambda\in \spec(G)} \mathbf{1}_{\{E_\lambda\not\subseteq \s^\perp\}},
\]
where $E_\lambda$ denotes the $\lambda$-eigenspace of $A(G)$ and $m_\lambda = \dim(E_\lambda)$. Thus $M(\s)$ is the number of main eigenvalues of $G^{\s}$, counted with
multiplicity, and $M_*(\s)$ is the number of \emph{distinct} main eigenvalues of $G^{\s}$. 

Conjectures \ref{conj:switching_n_main_eigenvalues} and \ref{conj:switching_n_main_eigenvectors} are then equivalent to showing that for any unsigned connected graph $G \notin\{ K_2, K_4 - e\}$, there is a switching $\s$ such that $M(\s) = n$. Equivalently, for any unsigned connected graph $G \notin\{ K_2, K_4 - e\}$, there is a switching $\s$ such that $M_*(\s) = |\spec(G)|$.

We prove asymptotic versions of the switching conjecture in the following sense.

\begin{theorem}\label{thm:asymptotic}
Let $G$ be a graph of order $n$.  Then there exists a switching
$\s\in\{\pm1\}^n$ such that
\[
M(\s)\ge  n - O\!\left(\frac{n}{(\log n)^{1/4}}\right)= n(1-o(1)).
\]
\end{theorem}

We prove Theorem \ref{thm:asymptotic} in Section \ref{sec:asymptotic}. A similar argument also shows the following.

\begin{theorem}\label{thm:asymptotic_distinct}
Let $G$ be a graph with $d\ge 2$ distinct eigenvalues. Then there exists a switching
$\s\in\{\pm1\}^n$ such that
\[
M_*(\s)\ge  d - O\!\left(\frac{d}{(\log d)^{1/4}}\right).
\]
\end{theorem}

Note that these two asymptotic results are incomparable. 

\subsection{Further discussion}

Perhaps not surprisingly, Conjecture \ref{conj:switching_n_main_eigenvalues} remains open for regular graphs. Another interesting unresolved case is the class of trees. We believe these families merit special attention and a proof for these special cases may even lead to a complete proof of the full conjecture. We pose this as a problem.

\begin{problem}
 Show that the switching conjecture for main eigenvalues holds for trees and connected regular graphs of order $n\ge 3$.
\end{problem}

A worthwhile relaxation of Conjecture \ref{conj:switching_n_main_eigenvalues} is to allow all signatures instead of just signatures arising from switchings. 

\begin{conjecture}[Signature conjecture for main eigenvalues]\label{conj:signature}
For any unsigned connected graph $G$ of order $n\ge 3$, there exists a signature $\sigma$ such that all eigenvalues of the signed graph $(G,\sigma)$ are main.
\end{conjecture}

It is known that all signatures of a tree can be obtained through switchings \cite[Proposition 3.2]{Zaslavsky_1982}, and so the trees remain a difficult case even for this relaxed claim.

Finally, we observe that while the original switching conjecture asks for non-orthogonality with the all-ones vector $\1$, the conjecture can be easily generalized to any Rademacher vector. In particular, Conjectures \ref{conj:switching_n_main_eigenvalues} and \ref{conj:switching_n_main_eigenvectors} are equivalent to the following seemingly stronger claim.

\begin{conjecture} Fix a Rademacher vector $\b\in \{\pm 1\}^n$. For any unsigned connected graph $G \notin\{ K_2, K_4 - e\}$ of order $n$, there is a  switching $\s$ such that $G^{\s}$ has an orthonormal eigenbasis consisting of eigenvectors which are all non-orthogonal to $\b$. 
\end{conjecture}

This is because, if there is a switching $\s$ such that $G^{\s}$ has an eigenbasis of main vectors, then $G^{\s\circ \b}$ has an eigenbasis of vectors which are non-orthogonal to $\b$; here, $\s\circ \b$ denotes the entrywise product of $\s$ and $\b$. Thus, all of our results carry forward to this more general setup.

\section{Asymptotic results}
\label{sec:asymptotic}

In this section, we prove Theorems \ref{thm:asymptotic} and \ref{thm:asymptotic_distinct}.

Let $N(\s)$ denote the number of non-main eigenvalues of $G^{\s}$, i.e., 
\[ N(\s) := \sum_{\lambda\in \spec(G)} m_\lambda \mathbf{1}_{\{E_\lambda\subseteq \s^\perp\}} = n-M(\s).\]
To prove Theorem \ref{thm:asymptotic}, it suffices to show that there is a switching $\s\in \{\pm 1\}^n$ such that 
\[N(\s) = O\!\left(\frac{n}{(\log n)^{1/4}}\right).\]
To that end, it is sufficient to show that for a random switching $\s\in \{\pm 1\}^n$ chosen uniformly, the expected value
\begin{equation}\label{eq:expected_non_main_eigenvalues}
   \mathbb{E}(N(\s)) = O\!\left(\frac{n}{(\log n)^{1/4}}\right). 
\end{equation}

To prove \eqref{eq:expected_non_main_eigenvalues}, we will classify the non-main eigenvalues of $G^{\s}$ into two categories: eigenvalues with at least one eigenvector of small support and eigenvalues with all eigenvectors with large support. We bound the number of non-main eigenvalues in each category separately and then show that the expected sum of those bounds is $O\!\left(\frac{n}{(\log n)^{1/4}}\right).$  

Fix an integer $L\ge 1$ which will be determined later. Define
\[
        \Lambda_{L}
        :=
        \{\lambda\in\spec(G):\min\{|\supp(\x)|:0\ne \x\in E_\lambda\}\le L\},
\]
i.e., $\Lambda_{L}$ contains all (distinct) eigenvalues of $G^{\s}$ which have an eigenvector whose support contains at most $L$ vertices. Let \[\Lambda_L^c:= \spec(G)\backslash \Lambda_L.\] 
Then, it is clear that 
\[ \mathbb{E}(N(\s)) = \sum_{\lambda\in \Lambda_L} m_\lambda \mathbb{P}(\s\perp E_\lambda) + \sum_{\lambda\in \Lambda_L^c} m_\lambda \mathbb{P}(\s\perp E_\lambda).\]

We recall the following consequence of a result of Erd\H{o}s \cite{Erdos_1945} concerning the Littlewood-Offord problem, which found a place in the \emph{Proofs from the Book} \cite{Aigner_Ziegler_2018}.

\begin{theorem}[\cite{Erdos_1945, Aigner_Ziegler_2018}]\label{thm:Littlewood_Offord}
If $\x\in\mathbb{R}^n$ has at least $\ell$ nonzero coordinates, and $\s$ is a random switching, then
\[
        \mathbb{P}(\s^\top \x=0)
        \le
        2^{-\ell}\binom{\ell}{\lfloor \frac{\ell}{2}\rfloor}. 
\]
Equivalently, the number of Rademacher vectors $\s\in \{\pm 1\}^n$ orthogonal to $\x$ is at most $2^{n-\ell}\binom{\ell}{\lfloor \frac{\ell}{2}\rfloor}$.
\end{theorem}

We use the above result to estimate the number of non-main eigenvalues in $\Lambda_L^c$.

\begin{claim}\label{claim:large_support} We have 
\[\sum_{\lambda\in \Lambda_L^c} m_\lambda \mathbb{P}(\s\perp E_\lambda)\le n\, O\left(\frac{1}{\sqrt{L}}\right).\]
\end{claim}

\begin{proof}
    If $\lambda\in \Lambda_L^c$, it means that every (non-zero) eigenvector $\x\in E_\lambda$ has $|\supp(\x)|\ge L+1$. Pick a $\lambda$-eigenvector $\x$. Using Theorem \ref{thm:Littlewood_Offord}, we see that 
    \begin{align*}
        \sum_{\lambda\in \Lambda_L^c} m_\lambda \mathbb{P}(\s\perp E_\lambda)& \le \sum_{\lambda\in \Lambda_L^c} m_\lambda \mathbb{P}(\s^\top \x = 0)\\
        & \le \left(\sum_{\lambda\in \Lambda_L^c} m_\lambda\right) 2^{-(L+1)}\binom{L+1}{\lfloor \frac{L+1}{2}\rfloor} \\
        & = n \, O\left(\frac{1}{\sqrt{L}}\right),
    \end{align*}
    where the last inequality follows from Stirling's approximation. 
\end{proof}

Next, we estimate the number of non-main eigenvalues in $\Lambda_L$.

\begin{claim}\label{claim:small_support} We have
    \[\sum_{\lambda\in \Lambda_L} m_\lambda \mathbb{P}(\s\perp E_\lambda)\le L^2\, 2^{L\choose 2}. \]
\end{claim}

\begin{proof}
    Consider a (non-zero) eigenvector $\x\in E_\lambda$ with minimum support. Since $\lambda \in \Lambda_L$, we see that $|\supp(\x)|\le L$. It is clear that $\lambda$ is an eigenvalue of the induced subgraph $G[\supp(\x)]$ with eigenvector $\x|_{\supp(\x)}$ (restriction of $\x$ to $\supp(\x)$). In particular, this means that $\lambda$ is an eigenvalue of some graph on at most $L$ vertices. We see that $|\Lambda_L|\le |\mathcal{E}(L)|$, where 
    \[\mathcal{E}(L):=\{\lambda \in \mathbb{R}: \lambda \text{ is an eigenvalue of some graph on at most $L$ vertices} \}.\] 

    Furthermore, it is well-known that a space of dimension $k$ contains at most $2^k$ Rademacher vectors. Thus, the number of Rademacher vectors in the space $E_{\lambda}^\perp$ is at most $2^{n-m_\lambda}$. This implies
    \[ \mathbb{P}(\s\perp E_\lambda)
        \le
        2^{-m_\lambda}.\]
    Using the above observations, we conclude that 
    \begin{align*}
        \sum_{\lambda\in \Lambda_L} m_\lambda \mathbb{P}(\s\perp E_\lambda) & \le  \sum_{\lambda\in \Lambda_L} m_\lambda\, 2^{-m_\lambda}\\
        & \le \frac{1}{2}|\Lambda_L|\\
        & \le \frac{1}{2}|\mathcal{E}(L)|\\
        & < \sum_{k=1}^L k\,2^{k\choose 2}\\
        & \le L^2\, 2^{L\choose 2}. 
    \end{align*}
The second last inequality above holds because there are at most $2^{k\choose 2}$ labelled graphs on $k$ vertices, each of which can have at most $k$ distinct eigenvalues. This completes the proof of the claim.
\end{proof}

Using Claims \ref{claim:large_support} and \ref{claim:small_support} we conclude that 
\[\mathbb{E}(N(\s)) \le n\, O\left(\frac{1}{\sqrt{L}}\right) + L^2\, 2^{L\choose 2}. \]
Since $L$ was arbitrary, we take $L=\left\lfloor \sqrt{\log_2 n}\right\rfloor,$ which implies
\[\mathbb{E}(N(\s)) \le O\left(\frac{n}{(\log n)^{1/4}}\right) + O
(\sqrt{n} \log n) =   O\left(\frac{n}{(\log n)^{1/4}}\right).\]
The proof of Theorem \ref{thm:asymptotic} is complete.

One can repeat the above argument with 
\[N_*(\s):= \sum_{\lambda\in \spec(G)} \mathbf{1}_{\{E_\lambda\subseteq \s^\perp\}} =d-M_*(\s)\]
to prove Theorem \ref{thm:asymptotic_distinct}. Here, $d = |\spec(G)|\ge 2$. Indeed, it is seen that 
\[\sum_{\lambda\in \Lambda_L^c} \mathbb{P}(\s\perp E_\lambda)\le d\, O\left(\frac{1}{\sqrt{L}}\right)\]
and 
 \[\sum_{\lambda\in \Lambda_L} \mathbb{P}(\s\perp E_\lambda)\le L^2\, 2^{L\choose 2}. \]
Thus,
\[\mathbb{E}(N_*(\s)) \le d\, O\left(\frac{1}{\sqrt{L}}\right) + L^2\, 2^{L\choose 2}. \]
Choosing $L=\left\lfloor \sqrt{\log_2 d}\right\rfloor$ gives the desired result Theorem \ref{thm:asymptotic_distinct}.

\section*{Acknowledgements}

Bojan Mohar is supported in part by the NSERC Discovery Grant R832714 (Canada), by the ERC Synergy grant (European Union, ERC, KARST, project number 101071836), and by the Research Project N1-0218 of ARIS (Slovenia). The research visit of Saieed Akbari at Simon Fraser University was supported in part by the ERC Synergy grant (European Union, ERC, KARST, project number 101071836). The authors thank Thom\'{a}s Jung Spier for helpful comments.

\section*{Declaration of AI use}

The authors acknowledge the use of ChatGPT (GPT-5.5, OpenAI; accessed July 2026) solely for preliminary brainstorming and the exploration of possible proof strategies. AI tools were not used to draft the manuscript. All formal statements, arguments, and proofs in the manuscript were written and checked by the authors, who take full responsibility for the accuracy and integrity of the article.

\bibliographystyle{plain}
\bibliography{switching_references}

\vspace{0.4cm}

\affl{Saieed Akbari}{s\_akbari@sharif.edu}{Department of Mathematical Sciences, Sharif University of Technology, Tehran, Iran}

\affl{Hitesh Kumar}{hitesh.kumar.math@gmail.com, hitesh\_kumar@sfu.ca}{Department of Mathematics, Simon Fraser University, Burnaby, Canada}

\affl{Bojan Mohar}{mohar@sfu.ca}{Department of Mathematics, Simon Fraser University, Burnaby, Canada\\On leave from FMF, Department of Mathematics, University of Ljubljana.}
 
\affl{Shivaramakrishna Pragada}{shivaramakrishna\_pragada@sfu.ca}{Department of Mathematics, Simon Fraser University, Burnaby, Canada}
\end{document}